\documentclass[12pt]{amsart} 

 \usepackage{caption}
 \usepackage{subcaption}

\usepackage[dvipsnames,usenames]{xcolor} 

\usepackage{amsfonts,latexsym,wasysym, mathrsfs, mathtools,hhline,color, bm}
\usepackage[all, cmtip]{xy}

\usepackage{url}

\definecolor{hot}{RGB}{65,105,225} 

\usepackage[pagebackref=true,colorlinks=true, linkcolor=hot ,  citecolor=hot, urlcolor=hot]{hyperref}

\usepackage{textcomp}
\usepackage{tipa}

\usepackage{graphicx,enumerate}
\usepackage{enumitem}
\usepackage[normalem]{ulem}

 \usepackage{amsthm}  
 \usepackage{amsmath} 
 \usepackage{amsfonts}  
 \usepackage{amssymb}
 \usepackage{amscd} 
 \usepackage[all]{xy} 

\usepackage{graphicx}
\usepackage{pdfsync}
\usepackage{url}
\usepackage{comment}

\usepackage[linesnumbered,ruled,vlined]{algorithm2e}
\usepackage{cleveref}

\numberwithin{equation}{section}
\newtheorem{theorem}{Theorem}
\newtheorem{proposition}[theorem]{Proposition}
\newtheorem{lemma}[theorem]{Lemma}
\newtheorem{corollary}[theorem]{Corollary}

\theoremstyle{definition}

\newtheorem{example}[theorem]{Example}

\numberwithin{theorem}{section}

\renewcommand{\P}{\mathbb{P}}
\newcommand{\R}{\mathbb{R}}
\newcommand{\Q}{\mathbb{Q}}
\newcommand{\C}{\mathbb{C} }

\newcommand{\newt}{{\mathrm{Newt}}}
\newcommand{\conv}{{\mathrm{Conv}}}
\newcommand{\mvol}{{\mathrm{MVol}}}
\newcommand{\vol}{{\mathrm{Vol}}}

\title[A polyhedral homotopy algorithm for computing critical points]{A polyhedral homotopy algorithm for computing critical points of polynomial programs}

\author[Lindberg]{Julia Lindberg}
\address{
Julia Lindberg \\
University of Texas-Austin \\ USA
}
\email{julia.lindberg@math.utexas.edu}
\urladdr{\url{https://sites.google.com/view/julialindberg/home}}

\author[Monin]{Leonid Monin}
\address{
Leonid Monin \\
EPFL, Switzerland
}
\email{leonid.monin@epfl.ch}
\urladdr{\url{http://www.math.toronto.edu/lmonin/}}

\author[Rose]{Kemal Rose}
\address{
Kemal Rose\\
MPI-MIS \\
 Germany
} 
\email{krose@mis.mpg.de}
\urladdr{\url{https://kemalrose.github.io/}}

\begin{document}

\begin{abstract}
In this paper we propose a method that uses Lagrange multipliers and numerical algebraic geometry to find all critical points, and therefore globally solve, polynomial optimization problems. We design a polyhedral homotopy algorithm that explicitly constructs an optimal start system, circumventing the typical bottleneck associated with polyhedral homotopy algorithms. The correctness of our algorithm follows from intersection theoretic computations of the algebraic degree of polynomial optimization programs and relies on explicitly solving the tropicalization of a corresponding Lagrange system. We present experiments that demonstrate the superiority of our algorithm over traditional homotopy continuation algorithms.
\end{abstract}

\maketitle

\section{Introduction}

Polynomial programming is a class of mathematical programming that seeks to minimize a polynomial objective function subject to polynomial constraints. 
These are optimization problems of the form
\begin{align}
    \min_{x \in \mathbb{R}^n} \ f_0(x)  \quad \text{subject to} \quad F(x) = 0, \tag{Opt} \label{eq:opt}
\end{align}
where $F(x) = \{f_1(x),\ldots,f_m(x)\}$ and $f_i(x) \in \mathbb{R}[x_1,\ldots,x_n]$ are polynomials. Throughout this paper we use the standard multi-index notation for polynomials. Namely, we denote 
\[f(x) = \sum_{\alpha \in \mathcal{A}} c_\alpha x^\alpha\] 
where $\mathcal{A} \subset \mathbb{N}^n$ is the monomial support of $f$ and for $\alpha \in \mathbb{N}^n$, $x^\alpha := x_1^{\alpha_1}\cdots x_n^{\alpha_n}$.

Polynomial programs have broad modelling power and therefore have naturally arisen in many applications including signal processing, combinatorial optimization, power systems engineering and more \cite{tan2001the,poljak1995a,molzahn2019a}. 
In general, these problems are NP hard to solve \cite{vavasis1990quadratic} but there exist many solution techniques and heuristics to tackle \eqref{eq:opt}. Some popular examples include the moment/SOS hierarchy \cite{parrilo2003semidefinite,lasserre2000global,wang2021tssos,invariants2022lindberg}, local methods \cite{boyd2004convex,potra2000interior} and the method of Lagrange multipliers \cite{bertsekas2014constrained}. This work proposes solving \eqref{eq:opt} by using the method of Lagrange multipliers along with techniques from numerical algebraic geometry.

The method of Lagrange multipliers works by taking a constrained optimization problem and lifting it to a higher dimensional space and then considering an unconstrained optimization problem. Given a problem of the form \eqref{eq:opt} we define its \emph{Lagrangian} as
\[L(x,\lambda) = f_0(x) - \sum_{j=1}^m \lambda_j f_j(x). \]
The corresponding \emph{Lagrange system} is then defined as $\mathcal{L}_{f_0,F} = \{\ell_1,\ldots,\ell_n,f_1,\ldots,f_m\}$ where
\[ \ell_i = \frac{\partial}{\partial x_i} L =  \frac{\partial}{\partial x_i} (f_0 - \sum_{j=1}^m \lambda_j f_j).\]

The main idea behind using Lagrange multipliers is that smooth critical points of \eqref{eq:opt} are zeroes of $\mathcal{L}_{f_0,F}$. Therefore, if we find all $(x,\lambda) \in \mathbb{R}^{n+m}$ that satisfy $\mathcal{L}_{f_0,F}(x,\lambda) = 0$, we will find all smooth local critical points, and therefore (so long as the variety of $F(x) = 0$ is smooth) the global optimum.

For fixed $f_0,F$ the number of complex solutions $\mathcal{L}_{f_0,F}= 0$ is called the \emph{algebraic degree} of \eqref{eq:opt}. For generic $f_0,F$ a formula for the algebraic degree is given in \cite[Theorem 2.2]{MR2507133} as
\begin{align}
d_1 \cdots d_m S_{n-m}(d_0-1,d_1-1,\ldots,d_m-1) \label{eq: alg deg} \end{align}
where $d_i = \deg(f_i)$ and   \[S_{r}(n_1, \ldots,n_k) = \sum_{i_1+\ldots+i_k =r} n_1^{i_1}\cdots n_k^{i_k}.\]

The algebraic degree has also been defined and studied for other classes of convex optimization problems in 
\cite{MR2496496}
and \cite{MR2546336}. When $f_0$ is the Euclidean distance function, i.e., $f_0 = \lVert x - u \rVert_2^2$ for some $u \in \mathbb{R}^n$, then the number of complex critical points to \eqref{eq:opt} is called the \emph{ED degree} of $F$. The ED degree was first defined in \cite{DraismaTheEDD}. Since then, other work has bounded the ED degree of a variety \cite{MR3451425}, studied the ED degree for real algebraic groups \cite{baaijensRealAlgGroups}, Fermat hypersurfaces \cite{leeFermat}, orthogonally invariant matrices \cite{drusvyatskiyOrthogonally}, smooth complex projective varieties \cite{aluffiEDComplex}, the multiview variety \cite{maximMultiview} and when $F$ consists of a single polynomial \cite{breiding2020euclidean}. 

Similarly, when $f_0$ is the likelihood function then the number of complex critical points of \eqref{eq:opt} is called the \emph{ML degree}. 
The ML degree was first defined in \cite{MR2230921, HostenSolving} and since then the relationship between ML degrees and Euler characteristics \cite{MR3103064}, Euler obstruction functions~\cite{MR3686780} and toric geometry \cite{MR3907355, MR4103774, lindberg2021the} has been extensively studied. Further, the ML degree of various statistical models has also been considered \cite{MR2988436,manivel,MR4219257,MR4196404}.

More recently, the algebraic degree of \eqref{eq:opt} has been considered when $f_0,\ldots,f_m$ are defined by sparse polynomials. In this case, the algebraic degree may be less than the bound given in \eqref{eq: alg deg}. The authors in \cite{ourpaper} showed that in some situations, the algebraic degree is equal to the mixed volume of the corresponding Lagrange system.
One corollary of this result, as well as the analogous results for the ML degree and Euclicdean distance degree in \cite{ourpaper, lindberg2021the,breiding2020euclidean}, is that if $f_0,F$ have generic coefficients, then \emph{polyhedral homotopy} algorithms are optimal for solving the corresponding Lagrange system in the sense that exactly one path is tracked for each complex solution of $\mathcal{L}_{f_0,F} =0$. A downside of polyhedral homotopy algorithms is that there is a bottleneck associated with computing a start system. 
The work in this paper makes progress in this regard by explicitly designing a polyhedral homotopy algorithm for \eqref{eq:opt} when $m = 1$, circumventing the standard bottle neck. We see this paper as the first step and inspiration for an exciting new line of research, namely explicitly constructing optimal homotopy algorithms for specific parameterized polynomial systems of equations.

The results of this paper are organized as follows. In \Cref{sec:poly homotopy}, we review the main idea behind polyhedral homotopy. In \Cref{sec:hyper} we explicitly construct a polyhedral homotopy algorithm for the case when there exists a single constraint. In \Cref{sec: hyp refined} we generalize this result to when this constraint is sparse. We present numerical results which show that our algorithm outperforms existing polyhedral homotopy solvers in \Cref{sec: numerics} and explicitly compute the algebraic degree of a certain multiaffine polynomial program in \Cref{sec: multi}.

\section{Polyhedral homotopy continuation} \label{sec:poly homotopy}

Homotopy continuation algorithms are a broad class of numerical algorithms used for finding all isolated solutions to a square system of polynomial equations. Specifically, suppose you have a square system of polynomial equations 
\[F(x) = \{f_1(x),\ldots,f_n(x) \} = 0\] where $f_i \in \mathbb{R}[x_1,\ldots,x_n]$ and the number of complex solutions to $F(x) = 0$ is finite. Homotopy continuation works by tracking solutions from an `easy' system  of polynomial equations (called the \emph{start system}) to the desired one (called the \emph{target system}). This is done by constructing a \emph{homotopy}, 
\[H(t;x) : [0,1] \times \C^n \longrightarrow \C^n,\] 
such that
\begin{enumerate}
    \item $H(0 ;x) = G(x)$ and $H(1;x) = F(x)$,
    \item the solutions to $G(x) = 0$ are isolated and easy to find
    \item $H$ has no singularities along the path $t \in [0,1)$ and
    \item $H$ is sufficient for $F$.
\end{enumerate}

Here we call a homotopy $H$ \emph{sufficient} for $F = H(1; x)$ if, 
by solving the ODE initial value problems $\frac{\partial H}{\partial t} + \frac{\partial H}{\partial x} \dot{x} = 0$
with initial values $\{ x \ : \ G(x) = 0\}$, 
all isolated solutions of $F(x) = 0$ can be obtained.

One example of a homotopy, known as a \textit{straight line homotopy}, is defined as a convex combination of the start and target systems:
\begin{align*}
    H(t;x) &= \gamma (1-t)G(x) + t F(x)
\end{align*}
where $\gamma \in \mathbb{C}$ is a generic constant. Choosing generic $\gamma$ ensures $H(x;t)$ is non-singular for $t \in [0,1)$.
Path tracking is typically done using standard predictor-corrector methods. For more information, see \cite{bertini-book, Sturmfels-CBMS}.

The main question when employing homotopy continuation techniques is how to select such an `easy' start system. If the target system roughly achieves the \emph{Bezout bound} then a \emph{total degree} start system is suitable. An example of this is
\[G(x) = \{x_1^{d_1} -1,\ldots,x_n^{d_n} -1 \} \]
where $\deg(f_i) = d_i$.

Often in applications, the target system is defined by sparse polynomial equations. In this case, the Bezout bound can be a strict upper bound on the total number of complex solutions so using a total degree start system leads to wasted computation. A celebrated result, known as the \emph{BKK bound}, gives an upper bound on the number of complex solutions in the torus to a sparse polynomial system. In order to state the BKK bound, we need a few preliminary definitions but recommend  \cite{mvoltext} for more details.

Given a polynomial $f = \sum_{\alpha \in \mathcal{A}} c_\alpha x^\alpha \in \mathbb{C}[x_1,\ldots,x_n]$ the \emph{Newton polytope} of $f$ is 
\[ \newt(f) = \conv \{\alpha \ : \alpha \in \mathcal{A}\}. \]
Given convex polytopes $P_1,\dots,P_n \subset \mathbb{R}^n$, consider the Minkowski sum 
$\mu_1 P_1 +\cdots +\mu_n P_n.$
A classic result shows that 
\[Q(\mu_1,\dots,\mu_n) = \vol(\mu_1 P_1 +\cdots +\mu_n P_n)\] 
is a homogeneous degree $n$ polynomial in $\mu_1,\dots, \mu_n$.
The \emph{mixed volume} of 
$P_1,\dots,P_n$ is the coefficient of $\mu_1 \cdots \mu_n$ of $Q$. We denote it as $\mvol(P_1,\ldots, P_n)$.

\begin{theorem}[BKK Bound \cite{bernshtein1979the,khovanskii1978newton,kouchnirenko1976polyedres}]\label{thm:BKK}
Let $F = \{ f_1,\dots,f_n \}$ be a  sparse polynomial system in
$\mathbb{C}[x_1,\dots,x_n]$ and let $P_1,\dots,P_n$ be their respective Newton polytopes. The number of isolated $\mathbb{C}^*$-solutions to $F=0$ is bounded above by $\mvol(P_1,\ldots,P_n)$. Moreover, if the coefficients of $F$ are general, then this bound is achieved with equality.
\end{theorem}

If the BKK bound is much less than the Bezout bound, a \emph{polyhedral} start system is a better choice since using a total degree start system will lead to wasted computation tracking homotopy paths that diverge to infinity. The downside of polyhedral homotopy is that the start system is more difficult to construct. This is not surprising since computing the mixed volume is $\#$P hard \cite{Khachiyan1993}. Even so, there is an algorithm that computes this start system \cite{huber1995a}. We briefly outline the idea behind polyhedral homotopy here but give \cite{huber1995a} as a more complete reference. 

Recall that $F = \{f_1,\ldots,f_n\}$, where $f_i = \sum_{\alpha \in \mathcal{A}_i}c_\alpha x^\alpha \in \mathbb{C}[x_1,\ldots,x_n]$. For each monomial, $\alpha \in \mathcal{A}_i$, we consider a \emph{lifting}, $w(\alpha)$, and the corresponding lifted system $F^w(x,t) = (f_1^w(x,t),\ldots,f_n^w(x,t))$ where
\begin{align}
 f_i(x,t) = \sum_{\alpha \in \mathcal{A}_i} c_\alpha x^\alpha t^{w(\alpha)}. \label{eq:lifted poly}
 \end{align}

Solutions to $F^w(x,t) = 0$ are algebraic functions in the parameter $t$. Such solutions can be written as
\[ x(t) = (x_1(t),\ldots,x_n(t)). \]

In a neighborhood of $t = 0$, each solution can be written as $x(t) = (x_1(t),\ldots,x_n(t))$ where
\[ x_i(t) = y_i t^{u_i} \  + \quad \text{higher order terms in } t  \]
where $y_i\neq 0$ is a constant and $u_i \in \mathbb{Q}$. Substituting this into \eqref{eq:lifted poly} we have
\[ f_i(x,t) = c_\alpha y^\alpha t^{u^T \alpha + w(\alpha)} \ + \quad \text{higher order terms in } t. \]

By \cite[Lemma 3.1]{huber1995a} We wish to find $u$ such that 
\[ \min_{u \in \mathbb{R}^n} \ \{ u^T \alpha + w(\alpha) \} \]
is achieved twice. For each solution $u$, the vector $(u,1)$ is an inner normal to one of the lower facets of the \emph{Cayley polytope} of $F$. Further more, each such solution, $u$, then induces a binomial polynomial system $\mathcal{B}_u$ which can be solved using Smith normal forms as well as a homotopy to track solutions from $\mathcal{B}_u(x) =0$ to $F(x)=0$. The sum of the number of solutions to $B_u(x) = 0$ for each solution $u$ is equal to the BKK bound of $F(x)$. Therefore, if the coefficients of $F$ are generic with respect to its monomial support, then polyhedral homotopy will track one homotopy path for each solution to $F(x) = 0$.
We illustrate this on a small example.

\begin{example}\label{ex:1}
Consider the system of one polynomial equation in one unknown
\[ f(x) = x^3-x^2+2x-1=0. \]
We wish to solve this polynomial system using homotopy continuation and a polyhedral start system. To do this we consider a lifted system of $f$ which we obtain by weighting each monomial of $f$ by some power of $t$:
\[ f_t = t^{\omega_3} x^3 - t^{\omega_2}x^2 +2t^{\omega_1}x - t^{\omega_0}.  \]
Now suppose we choose weighting $(\omega_0, \omega_1,\omega_2, \omega_3) = (0,3,1,2)$ so 
\[ f_t = t^2x^3 - tx^2 + 2xt^3 - t^0.  \]
A figure of this lifting is given in Figure~\ref{fig:ex1}. Solutions to $f_t = 0$ lie in the field of Puiseux series of $t$ and are of the form
\[x(t) = \hat{x} t^a + \text{ higher order terms in } t\]
where $a \in \Q$ and $\hat{x} \in \mathbb{C}^*$. For $x(t)$ to be a root of $f_t$, the lowest terms in $t$ must cancel out. Substituting in $x(t) = \hat{x} t^a $ into $f_t$, we have
\begin{align}
f_t(x(t)) = \hat{x}^3 t^{3a+2} - \hat{x}^2t^{2a+1} + 2 \hat{x} t^{a+3} - t^0. \label{eq_ft}
\end{align}
To have cancellation of the lowest terms, we must have the minimum exponent in $t$ achieved twice. In other words,
\begin{align}
 \min_a \{3a+2, 2a+1, a+3, 0\} \label{eq: trop ex}.
\end{align}
must be achieved twice. There are six options:
\begin{enumerate}
    \item $3a+2=2a+1 <a+3,0$
    \item $3a+2=a+3<2a+1,0$
    \item $3a+2 = 0<2a+1,a+3$
    \item $2a+1 = a+3<3a+2,0$
    \item $2a+1 = 0<3a+2,a+3$
    \item $a+3 = 0 < 3a+2, 2a+1$
\end{enumerate}
The only feasible solutions are the first and fifth where $a = -1$ and $a = -\frac{1}{2}$, respectively. For the first case, we substitute $a = -1$ into \eqref{eq_ft} giving
\[ \hat{x}^3 t^{-1} - \hat{x}^2 t^{-1} + 2 \hat{x}t^2 -1  \]
Multiplying through by $t$, we get
\[h_1(\hat{x},t) = \hat{x}^3 - \hat{x}^2 + 2 \hat{x}t^3 - t. \]
When $t = 0$ we have $h_1(\hat{x},0) = \hat{x}^3 - \hat{x}^2$ which has a unique $\mathbb{C}^*$ solution, $\hat{x} = 1$.

Similarly, we consider when $a = -\frac{1}{2}$ and substitute this value of $a$ into \eqref{eq_ft} to get 
\[ h_2(\hat{x},t) = \hat{x}^3 t^{\frac{1}{2}} - \hat{x}^2 + 2 \hat{x}^{\frac{7}{2}} - 1. \]
When $t = 0$ we have $h_2(\hat{x},0) = - \hat{x}^2 -1$ which has two $\mathbb{C}^*$ solutions, $\hat{x} = \pm \sqrt{-1}$. Therefore, to find all three solutions to $f(x) = 0$, we track the solution $\hat{x} = 1$ using the homotopy $h_1(\hat{x},t)$ from $t = 0$ to $t = 1$ and the solutions $\hat{x} = \pm \sqrt{-1}$ using the homotopy $h_2(\hat{x},t)$ from $t = 0$ to $t = 1$. A graphical depiction of the homotopy $h_1$ is shown in Figure~\ref{fig:homotopyex}.

\begin{figure}
    \centering
    \includegraphics[width = 0.3\textwidth]{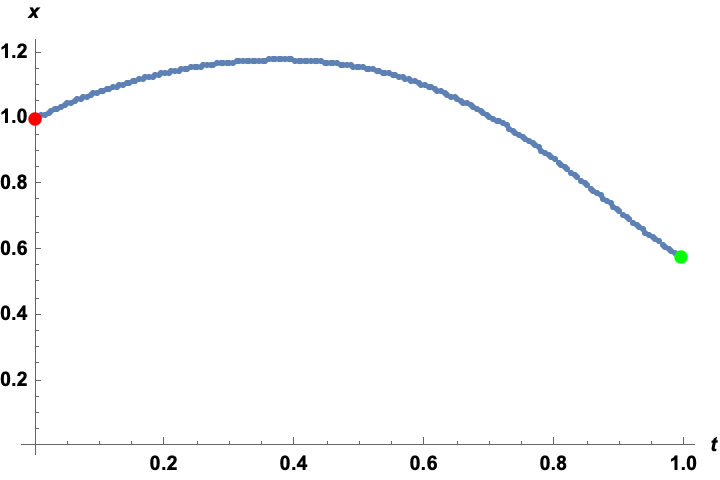}
    \caption{The homotopy $h_1(\hat{x},t)$ from Example~\ref{ex:1}. The red point is the starting point induced by the binomial system $\hat{x}^3 - \hat{x}^2 = 0$ while the green point is the target solution, namely a zero of $f(x) = 0$. }
    \label{fig:homotopyex}
\end{figure}

\begin{figure}
    \centering
    \includegraphics[width = 0.3\textwidth]{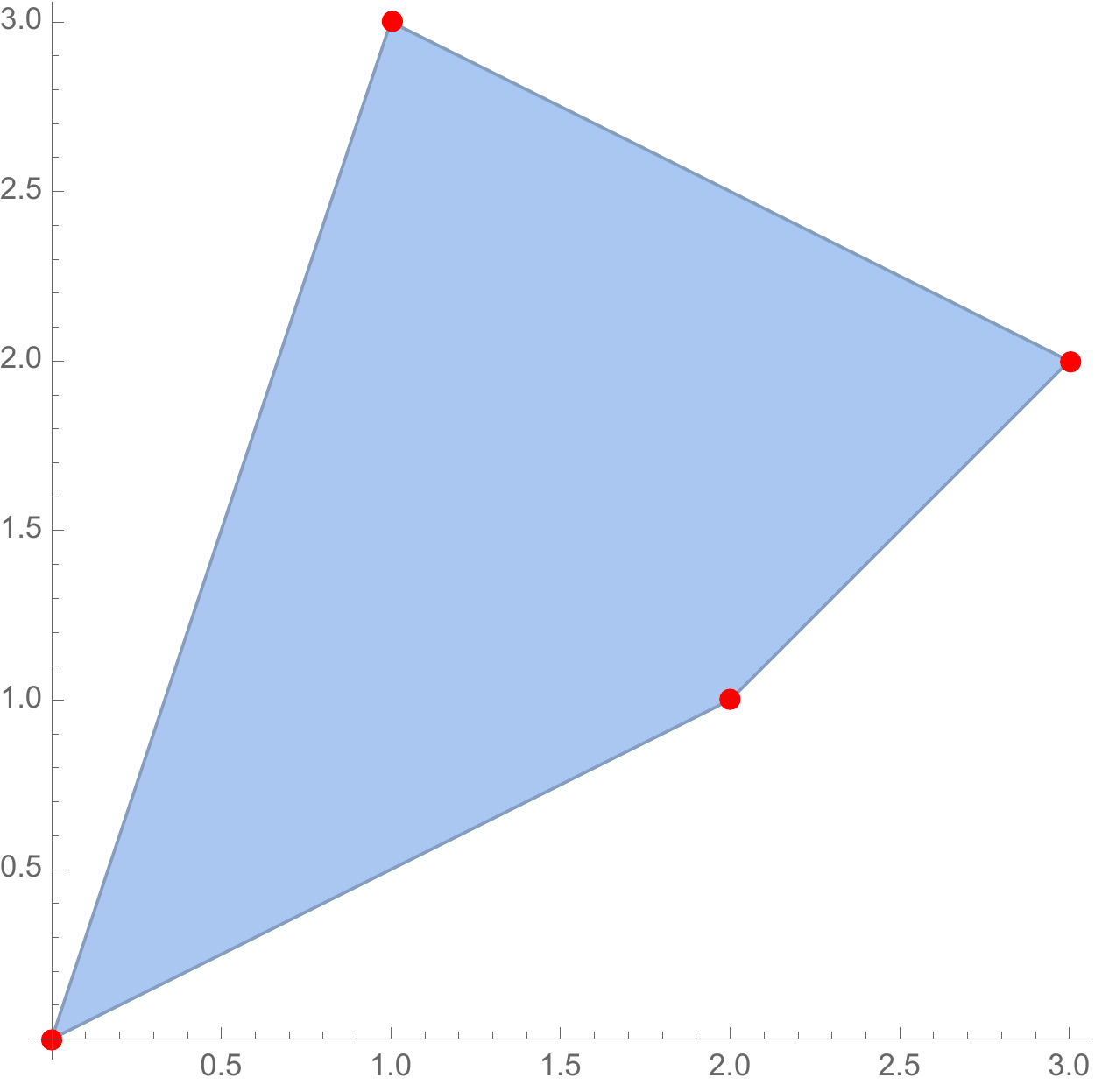}
    \caption{The polyhedral lift from Example~\ref{ex:1}}
    \label{fig:ex1}
\end{figure}

Finally, one can observe in \Cref{fig:ex1} that the lifted polytope of $\newt(f)$ has two lower facets, $\mathcal{F}_1 = \conv\{(0,0),(2,1)\}$ and $\mathcal{F}_2 = \conv\{(2,1),(3,2)\}$. $\mathcal{F}_1$ has inner normal given by $(-\frac{1}{2},1)$ while $\mathcal{F}_2$ has inner normal given by $(-1,1)$. These are precisely the solutions to \eqref{eq: trop ex}.
\end{example}

The main bottleneck with employing polyhedral homotopy algorithms is finding the binomial start systems and corresponding homotopies. Example~\ref{ex:1} shows how finding these start systems is equivalent to solving a tropical system for a fixed lifting. The main contribution of this paper is to find these binomial start systems for polynomial systems arising as the Lagrange systems of polynomial optimization programs.

\section{General hypersurface}\label{sec:hyper}
\label{section:general hypersurface}

We consider \eqref{eq:opt} when $m = 1$ and $\deg(f_0) = 1$. Specifically, we consider a polynomial optimization problem of the form 
\begin{align}
\min_{x \in \R^n} \ u^T x \quad \text{s.t.} \quad  f(x) = 0 \label{eq: hyp1}
\end{align}
where $u \in \R^n$ and $f(x)$ is a general degree $d\geq 2$ polynomial.
We wish to design a homotopy algorithm to find all critical points to \eqref{eq: hyp1}. We first consider the Lagrange system $\mathcal{L}_{u,f} = \{\ell_1,\ldots, \ell_n, f\}$ of \eqref{eq: hyp1} where
\begin{equation}
    \begin{aligned}
    \ell_i = u_i - \lambda \frac{\partial}{\partial x_i} f(x).
    \end{aligned}
\end{equation}

If $f$ is a generic degree $d$ polynomial and $u \in \mathbb{R}^n$ is generic, then by \cite{ourpaper}, the number of critical points to \eqref{eq: hyp1} is the same as that of 
\begin{align}
  \min_{x \in \R^n} \ u^T x \quad \text{s.t.} \quad  \hat{f}(x) = 0 \label{eq: hyp2}  
\end{align}
where $\hat{f} = \sum_{i=1}^n c_i x_i^d$ and $c_i$ is generic for $i \in [n]$. 
 The Lagrange system of \eqref{eq: hyp2} is $\mathcal{L}_{u,\hat{f}} = \{\hat{\ell}_1,\ldots,\hat{\ell}_n, \hat{f}\}$ where for $i \in [n]$ 
\begin{equation}
    \begin{aligned}\label{eq: lag hyp}
    \hat{\ell}_i &= u_i - d \lambda c_i x_i^{d-1} 
    \end{aligned}
\end{equation}
Observe that by \cite{ourpaper}, not only are the algebraic degrees of $(u,f)$ and $(u, \hat{f})$ the same, but the BKK bound of $\mathcal{L}_{u,\hat{f}}$ is the same as that of $\mathcal{L}_{u,f}$.

The Lagrange system $\mathcal{L}_{u,\hat{f}}$ is sparser than $\mathcal{L}_{u,f}$ and and in fact a binomial start system $G$
for $\mathcal{L}_{u,\hat{f}}$ can be constructed efficiently.
The following lemma shows that this is desirable since start systems for
$\mathcal{L}_{u,\hat{f}}$ are start systems for $\mathcal{L}_{u,f}$ as well. We first need an observation about the existence of straight line homotopies.

\begin{proposition}
\label{prop: straight-line homotopies}
Let $F(x; p) : \C^n \times C^k \longrightarrow \C^n$
denote a family of polynomials systems $F(x; p)$ that depends polynomially on parameters $p \in \C^k$
and $F(x; p_1)$ a fixed member of that family.
Then there is a nonempty set $U \subseteq \C^k$, open and dense in the Euclidean topology, such that
for every parameter $p_0$ in $U$ the straight line homotopy
\[
 H(t;x) =  (1-t) F(x; p_1) + t F(x; p_0)
\]
is sufficient for $F(x; p_1)$.
\end{proposition}
\begin{proof}
By the \emph{Parameter Continuation Theorem} by Morgan and Sommese \cite{sommese2005numerical}
there exists a proper algebraic subvariety $\Sigma \subset \C^k$ with the following property: 
Let $\rho: [0, 1] \rightarrow \C^k$ be any smooth path
and $H(t, x)  = F(x,\rho(t))$ the
corresponding homotopy.
If $$\rho([0,1)) \cap \Sigma = \emptyset ,$$ then as $t~\rightarrow~1$, the limits of the solution paths
$x(t)$
satisfying $H(x(t), t) = 0$ include all the isolated solutions to $F(x; \rho(1))$ = 0.
In particular, $H(t, x)$ is a sufficient homotopy.

From now on we identify the complex affine space $\C^k$ with real affine space $\R^{2k}$ and denote by
$\overline{\Sigma}$ the closure of $\Sigma$ in real projective space $\P_{\R}^{2k+1}$.
Consider the projection $\pi : \P_{\R}^{2k+1} 
\dashrightarrow \P_{\R}^{2k}$ away from the point $p_1$.
Since the codimension of $\overline{\Sigma}$, considered as a manifold, is at least two,
the image $\pi \left( \overline{\Sigma} \right)$ has codimension at least one in $\P_{\R}^{2k}$.
In particular,
the image $\pi(p_0)$ of a generic element $p_0$ is not contained in $\pi \left( \overline{\Sigma} \right)$.
Since the image $\rho(  [0,1) )$ of the
straight path
$$
    \rho(t) = (1-t)  p_1 + t p_0
$$
between $p_0$ and $p_1$ is is contained in the fiber $\pi^{-1}(\pi(p_0))$, it
does not intersect $\Sigma$. Consequently the to $\rho$ associated straight line homotopy is sufficient.
\end{proof}

\begin{lemma}
\label{lem:hom}
Let $G$ be a zero dimensional quadratic system of polynomials with exactly
$BKK(\mathcal{L}_{u,\hat{f}})$ solutions.
There is a sufficient homotopy connecting 
$G$
to $\mathcal{L}_{u,f}$.
\end{lemma}
\begin{proof}
Let 
$
F(x; c) 
$
denote the family of polynomials with monomial support contained in the support of
$\mathcal{L}_{u,f}$. In particular, the coefficient vector $c$ has one entry for each monomial of each polynomial of $\mathcal{L}_{u,f}$.
We denote by $F(x; c_0)$ a generic member of this family.

The desired homotopy will be constructed explicitly as a composition. We start by
connecting $F(x; c_0)$ to both $\mathcal{L}_{u,f}$ and $G$ with a straight line homotopy, which
by Proposition \ref{prop: straight-line homotopies} is a sufficient homotopy in both cases.
We denote the straight line homotopy from $F(x; c_0)$ to $G$ by $H$.
It now suffices to prove that
$H$ does not merge any solutions of $F(x; c_0)$, allowing us to define the inverted homotopy $H^*$ by setting
for $t$ in $(0,1)$
$H^*(t,x) = H(1-t,x)$ and
$H^*(0,x) = G(x)$.
Since tracking the roots of $F(x; c_0)$ to the roots of $G$ along the sufficient homotopy $H$ defines a surjective map, it is enough to prove that $F(x; c_0)$ and $G$ have the same number of solutions.

By the results of Bernstein and Kushnirenko \cite{bernshtein1979the,kouchnirenko1976polyedres},
the number of solutions of $F(x; c_0)$ is equal to the BKK bound of $\mathcal{L}_{u,f}$.
In \cite{ourpaper} the authors prove
that the polynomial system $\mathcal{L}_{u,f}$ achieves this bound.
Furthermore, as we noted at the beginning of Section \ref{section:general hypersurface}, $\mathcal{L}_{u,f}$ and $\mathcal{L}_{u,\hat{f}}$ have the same number of solutions:
\begin{equation}
\label{eq: inequalities one}
    \# \{ \mathcal{L}_{u,\hat{f}} \} = 
\# \{ \mathcal{L}_{u,f} = 0 \} = 
BKK(\mathcal{L}_{u,\hat{f}}).
\end{equation}
At the same time the number of solutions to $G$ is equal to the BKK bound of 
$\mathcal{L}_{u,\hat{f}}$, which is upper bounded by the BKK bound of $\mathcal{L}_{u,f}$ by inclusion on Newton poytopes:
\begin{equation}
\label{eq: inequalities two}
    \# \{ \mathcal{L}_{u,\hat{f}} \} \leq 
\# \{ G = 0 \} \leq   BKK(\mathcal{L}_{u,\hat{f}}).
\end{equation}
Together, inequalities \eqref{eq: inequalities one} and \eqref{eq: inequalities two} imply that
$\mathcal{L}_{u,\hat{f}}$ and $G$ have the same root count.
\end{proof}

We now give the main result of this section.

\begin{theorem}\label{thm:hyp}
For any $d \geq 2$ consider the Lagrange system of \eqref{eq: hyp1}. Then for generic $u$ and $f$ there are $d(d-1)^{n-1}$ complex solutions to the corresponding Lagrange system. Moreover, all of these solutions can be found via the homotopy 
\[H(x, \lambda ; t) = (1-t)B(x, \lambda) + t \gamma \mathcal{L}_{u,f}(x, \lambda) \] 
where 
\begin{align}\label{eq:bin hyp}
    B(x,\lambda) &= \begin{cases} 
    u_1 - d \lambda c_1 x_1^{d-1} = 0 \\ 
    \ \vdots \\ 
    u_n - d \lambda c_n x_n^{d-1} = 0\\
    \ c_0 + c_1 x_1^d = 0,
    \end{cases}
\end{align}
$\gamma \in \mathbb{C}$ is a generic constant and $\mathcal{L}_{u,f}(x, \lambda) $ is the Lagrange system of \eqref{eq: hyp1}.
\end{theorem}

\begin{proof}

In order to design a polyhedral homotopy algorithm as described in \cite{huber1995a}, in the following
we construct a binomial start system $B$ of $\mathcal{L}_{u, \hat{f}}$ by solving a tropical system.
By  the proof of Lemma~\ref{lem:hom} we then obtain a homotopy from $B$ to 
$\mathcal{L}_{u, f}$.
Note that,
 by genericity of $f$, this homotopy can be chosen to be a straight line homotopy.

By Lemma~\ref{lem:hom} it suffices to design a polyhedral homotopy algorithm as described in \cite{huber1995a} for $\mathcal{L}_{u, \hat{f}}$. In order to define this algorithm, we need to first find a binomial start system of $\mathcal{L}_{u, \hat{f}}$ which can be done by solving a tropical system.

Let $a_i$ be the tropical variable corresponding to $x_i$ and $b$ the tropical variable corresponding to $\lambda$. Then for a given lifting $\omega \in \mathbb{R}^{3n+1}$, the corresponding tropical system that we want to solve is 
\begin{equation}\label{eq:trop general}
\begin{aligned}
    &\min_{a\in\mathbb{Q}^n, b \in \mathbb{Q}} \  \{\omega_{1,1}, (d-1)a_1 + b + \omega_{1,2}\} \\
    &\ \ \vdots \\
    &\min_{a\in\mathbb{Q}^n, b \in \mathbb{Q}} \  \{\omega_{n,1}, (d-1)a_n + b + \omega_{n,2}\} \\
    &\min_{a\in\mathbb{Q}^n, b \in \mathbb{Q}} \  \{\omega_{n+1,1}, d a_1 + \omega_{n+1,2},\ldots,da_n + \omega_{n+1,n+1}\}
\end{aligned}
\end{equation}

We consider a specific lifting that induces a unique solution to \eqref{eq:trop general}, giving a homotopy from one binomial start system to the desired target system \eqref{eq: lag hyp}.
With the particular lifting 
\begin{equation}\label{eq:lift hyp}
\begin{aligned}
    \omega_{ij} = \begin{cases} 
    0 & \text{if} \quad 1 \leq i \leq n+1, \ j = 1 \\
    1-d & \text{if} \quad 1 \leq i \leq n, \ j = 2 \\
    -d & \text{if}\quad  (i,j) = (n+1,2) \\
    1-d &\text{else}
    \end{cases}
\end{aligned}
\end{equation}
This gives the following tropical system:
\begin{equation}\label{eq:trop_hyp}
\begin{aligned}
    &\min_{a\in\mathbb{Q}^n, b \in \mathbb{Q}} \  \{0, (d-1)a_1 + b + 1-d\} \\
    &\ \qquad \vdots \\
    &\min_{a\in\mathbb{Q}^n, b \in \mathbb{Q}} \  \{0, (d-1)a_n + b + 1-d\} \\
    &\min_{a\in\mathbb{Q}^n, b \in \mathbb{Q}} \  \{0, d a_1 -d , d a_2 + 1-d, \ldots,da_n + 1-d\}
\end{aligned}
\end{equation}

We claim there is a unique solution to \eqref{eq:trop_hyp} given by $a_i = 1$ for $i \in [n]$ and $b = 0$.

First, observe that the first $n$ equations of \eqref{eq:trop_hyp} force $(d-1) a_i + b + 1 - d = 0$ for $i \in [n]$. This gives $a_i = \frac{d-1 -b}{d-1}$. Substituting this into the final equation and simplifying we have that 
\[\min_{a\in\mathbb{Q}^n, b \in \mathbb{Q}} \  \{0, \frac{bd}{1-d} , \frac{bd}{1-d}+1, \ldots,\frac{bd}{1-d}+1\} \]
must have minimum attained twice. It is then clear that the only solution is $b = 0$ where the minimum is achieved at the first two terms. Back substituting then gives that $a_i = \frac{d-1}{d-1} = 1$ for $i \in [n]$. The binomial start system $\mathcal{B}(x, \lambda)$ defined in \eqref{eq:bin hyp} then follows immediately from the solution to this tropical system.
\end{proof}

Observe that Bezout's Theorem gives an upper bound that \eqref{eq: lag hyp} has at most $d^{n+1}$ solutions but we see that the binomial system \eqref{eq:bin hyp} has $d(d-1)^{n-1}$ solutions. This gives another proof of the bound given in \cite{nie2009algebraic} for hypersurfaces and highlights the benefit of using a polyhedral start system over a total degree start system.

Finally, we wish to remark that homotopy defined in \Cref{thm:hyp} will work for finding all smooth critical points for the optimization of a linear function over any hypersurface, $f$, so long as $\newt(f)$ is contained in $\conv\{0,de_1,\ldots,d e_n\}$. When $\newt(f)$ is a strict subset of $\conv\{0,de_1,\ldots,d e_n\}$, then algebraic degree of $f$ can be less than $d(d-1)^{n-1}$ meaning, this homotopy may lead to wasted computation in tracking divergent paths.

\section{Refined hypersurface}\label{sec: hyp refined}
We wish to now refine the hypersurface cased discussed in the previous section. Instead of assuming $f(x)$ generic degree $d$ hypersurface, we assume $\newt(f) = \conv \{ 0, d_1 e_1, \ldots, d_n e_n \}$. As above, to design an optimal binomial start system we first consider the monomials only corresponding to vertices of $\newt(f)$. In this case, we consider $f = c_0 + \sum_{i=1}^n c_i x_i^{d_i}$ where $c_i$ are generic constants. In this case, the Lagrange system corresponding to \eqref{eq: hyp1} is $\mathcal{L}_{u,f} = \{\ell_1,\ldots, \ell_n, f\}$ where for $i \in [n]$
\[ \ell_i = u_i - d_i c_i \lambda x_i^{d_i - 1}. \]

\begin{theorem}\label{thm:hyp refined}
Consider \eqref{eq: hyp1} where $u$ is generic and  
\[\newt(f) = \conv\{0,d_1e_1,\ldots,d_n e_n \}\]
where $1 \leq d_1 \leq d_2 \leq \cdots \leq d_n$ and the non-zero coefficients of $f$ are generic. 
The algebraic degree of \eqref{eq: hyp1} is
\[d_1 \cdot (d_2-1) \cdots (d_n - 1).\] 
Moreover, all solutions of $\mathcal{L}_{u,f}(x) = 0$ can be found via the homotopy $H(x, \lambda ; t) = (1-t)B(x, \lambda) + \gamma t \mathcal{L}_{u,f}(x, \lambda) $ where 
\begin{align}\label{eq:bin hyp 2}
    B(x,\lambda) &= \begin{cases} 
    u_1 - d_1 \lambda c_1 x_1^{d_1-1} = 0 \\ 
    \ \vdots \\ 
    u_n - d_n \lambda c_n x_n^{d_n-1} = 0\\
    \ c_0 + c_1 x_1^{d_1} = 0,
    \end{cases}
\end{align}
$\gamma \in \mathbb{C}$ is generic.
\end{theorem}

\begin{proof}
As before, we design a polyhedral homotopy algorithm as described in \cite{huber1995a} for $\mathcal{L}_{u, \hat{f}}$.

Let $a_i$ be the tropical variable corresponding to $x_i$ and $b$ the tropical variable corresponding to $\lambda$. Then for a given lifting $\omega \in \mathbb{R}^{3n+1}$, the corresponding tropical system that we want to solve is 
\begin{equation}\label{eq:trop general 2}
\begin{aligned}
    &\min_{a\in\mathbb{Q}^n, b \in \mathbb{Q}} \  \{\omega_{1,1}, (d_1-1)a_1 + b + \omega_{1,2}\} \\
    &\ \ \vdots \\
    &\min_{a\in\mathbb{Q}^n, b \in \mathbb{Q}} \  \{\omega_{n,1}, (d_n-1)a_n + b + \omega_{n,2}\} \\
    &\min_{a\in\mathbb{Q}^n, b \in \mathbb{Q}} \  \{\omega_{n+1,1}, d_1 a_1 + \omega_{n+1,2},\ldots,d_na_n + \omega_{n+1,n+1}\}
\end{aligned}
\end{equation}

We consider a specific lifting that induces a unique solution to \eqref{eq:trop general 2}, giving a homotopy from one binomial start system to the desired target system.
Consider the particular lifting 
\begin{equation}\label{eq:lift hyp 2}
\begin{aligned}
    \omega_{ij} = \begin{cases} 
    0 & \text{if} \quad 1 \leq i \leq n+1, \ j = 1 \\
    1-d_i & \text{if} \quad 1 \leq i \leq n, \ j = 2 \\
    -d_1 & \text{if}\quad  (i,j) = (n+1,2) \\
    1-d_i &\text{if}\quad i = n+1, \ 3 \leq j \leq n+1
    \end{cases}
\end{aligned}
\end{equation}
This gives the following tropical system:
\begin{equation}\label{eq:trop_hyp 2}
\begin{aligned}
    &\min_{a\in\mathbb{Q}^n, b \in \mathbb{Q}} \  \{0, (d_1-1)a_1 + b + 1-d_1\} \\
    &\ \qquad \vdots \\
    &\min_{a\in\mathbb{Q}^n, b \in \mathbb{Q}} \  \{0, (d_n-1)a_n + b + 1-d_n\} \\
    &\min_{a\in\mathbb{Q}^n, b \in \mathbb{Q}} \  \{0, d_1 a_1 -d_1 , d_2 a_2 + 1-d_2, \ldots,d_na_n + 1-d_n\}
\end{aligned}
\end{equation}

We claim there is a unique solution to \eqref{eq:trop_hyp 2} given by $a_i = 1$ for $i \in [n]$ and $b = 0$.

First, observe that the first $n$ equations of \eqref{eq:trop_hyp 2} force $(d_i-1) a_i + b + 1 - d_i = 0$ for $i \in [n]$. This gives $a_i = \frac{d_i-1 -b}{d_i-1}$. Substituting this into the final equation and simplifying we have that 
\begin{align}
\min_{a\in\mathbb{Q}^n, b \in \mathbb{Q}} \  \{0, \frac{bd_1}{1-d_1} , \frac{bd_2}{1-d_2}+1, \ldots,\frac{bd_n}{1-d_n}+1\} \label{eq:last trop}
\end{align}
must have minimum attained twice. It is clear that there is a solution when $b = 0$, where the minimum is achieved at the first two terms. Back substituting then gives that $a_i = \frac{d-1}{d-1} = 1$ for $i \in [n]$. The binomial start system $B(x, \lambda)$ defined in \eqref{eq:bin hyp} then follows immediately from the solution to this tropical system.

It remains to show that there are no other solutions to \eqref{eq:last trop}. There are three cases to rule out:
\begin{enumerate}
\item  the minimum of \eqref{eq:last trop} is not attained at $0$ and $\frac{b d_i}{1-d_i} +1$ for $2 \leq i \leq n$; 
\item the minimum of \eqref{eq:last trop} is not attained at $\frac{b d_i}{1-d_i} +1$ and $\frac{b d_1}{1-d_1}$ for $2 \leq i \leq n$; and 
\item the minimum of \eqref{eq:last trop} is not attained at $\frac{b d_i}{1-d_i} +1$ and $\frac{b d_j}{1-d_j} +1$ for $i \neq j$, $2 \leq i,j \leq n$ 
\end{enumerate}
For the first case, observe that if $0 = \frac{b d_i}{1-d_i} + 1$ for some $2 \leq i \leq n$, then $b = \frac{d_i -1}{d_i}$. This then implies that $\frac{b d_1}{1-d_1} = \frac{d_1}{1-d_1} \cdot \frac{d_i - 1}{d_i}<0$ so the minimum is not attained at $0$. To rule out case $(2)$, consider when $\frac{b d_i}{1-d_i} + 1 = \frac{bd_1}{1-d_1}$. If $d_i = d_1$ then there is no solution so suppose $d_i > d_1$. In this case, $b = \frac{(d_1 -1)(d_i -1)}{d_1 - d_i}<0$ and $\frac{d_1b}{1-d_1} = \frac{d_1(d_i-1)}{d_i-d_1}>0$ so the minimum would be attained at $0$ instead. Finally, if $\frac{b d_i}{1-d_i} +1 = \frac{b d_j}{1-d_j} +1$ this implies that $b = 0$ and in this case the minimum is attained at $0$ and $\frac{bd_1}{1-d_1}$.
\end{proof}

As a corollary we now have a families of hypersurfaces with algebraic degree one and zero.

\begin{corollary}\label{cor:alg deg 1}
Consider the Lagrange system of \eqref{eq: hyp1} where $u$ and $f$ are generic and $\newt(f) = \conv \{0, e_1, 2e_2,\ldots,2e_n \}$. Then the algebraic degree of $(u,f)$ is one.
\end{corollary}

We remark that this is the first instance that the authors are aware of that gives a partial classification of polynomial programs with algebraic degree one. This is in contrast to the ML degree, where \cite{MR3103064} classifies very affine varieties with ML degree one. It is an interesting open question to give a complete classification of polynomial programs with algebraic degree one.

\begin{example}
Consider the optimization problem
\begin{align}\label{ex: opt}
    \min_{x_1,x_2 \in \mathbb{R}} \ u_1 x_1 + u_2 x_2 \quad \text{s.t.} \quad c_0 + c_1 x_1 + c_2 x_2 +c_3 x_2^2 = 0
\end{align}
where $u_1,u_2,c_0,c_1,c_2,c_3 \in \mathbb{R}$ are real valued parameters. By \Cref{cor:alg deg 1},  \eqref{ex: opt} has algebraic degree one, meaning the Lagrange system
\begin{align*}
    u_1 - \lambda c_1 &=0 \\
    u_2 - \lambda(c_2 + 2 c_3 x_2) &=0 \\
    c_0 + c_1 x_1 + c_2 x_2 +c_3 x_2^2 &= 0
\end{align*}
has one solution. This solution can then be expressed as a rational function of the problem data $u_1,u_2,c_0,c_1,c_2,c_3$. In this case, the unique solution is
\[x_1 = \frac{c_2^2 u_1^2 - 4c_0c_3u_1^2 - c_1^2u_2^2}{4c_1c_3u_1^2}, \quad x_2 = \frac{ c_1 u_2 - c_2u_1}{2c_3u_1}, \quad  \lambda = \frac{u_1}{c_1}. \]
\end{example}

Similarly, \Cref{thm:hyp refined} also gives a family of polynomial programs with algebraic degree zero.

\begin{corollary}
Consider the Lagrange system of \eqref{eq: hyp1} where $u$ and $f$ are generic and $\newt(f) = \conv \{0, e_1,\ldots,e_k,2e_{k+1},\ldots,2e_n \}$ for some $2 \leq k \leq n$. Then the algebraic degree of $(u,f)$ is zero.
\end{corollary}

\section{Numerical results}\label{sec: numerics}

\begin{table}
    \centering
    \begin{tabular}{c|c|c|c|c|c|c|c|c}
         $n$&  $20$ & $30$ & $40$ & $50$ & $60$ & $70$ & $80$ & $90$  \\
         Polyhedral &    $0.14$  & $0.51$ &  $1.01$ & $2.30$ & $4.49$ & NA & NA & NA \\
        $H$  &  $0.07$ & $0.20$  & $0.35$ & $0.87$ & $1.65$ & $2.54$ & $3.78$ & $6.45$ \\ 
    \end{tabular}
    \caption{Average time (sec) to find all critical points to \eqref{eq:opt} when $d = 2$ using standard polyhedral homotopy versus the homotopy, $H$, outlined in Theorem~\ref{thm:hyp}.}
    \label{tab1}
\end{table}

\begin{table}
    \centering
    \begin{tabular}{c|c|c|c|c|c|c|c}
         $n$& $6$ & $7$ & $8$ & $9$ & $10$ & $11$ & $12$  \\
         Polyhedral &  $0.29$ & $0.93$ & $3.06$ & $9.79$ & $27.42$ & $88.37$ & $556.92$ \\
        $H$ &  $0.21$ & $0.68$ & $2.29$ & $7.35$ & $20.35$ & $70.02$ & $395.64$\\ 
    \end{tabular}
    \caption{Average time (sec) to find all critical points to \eqref{eq:opt} when $d = 3$ using standard polyhedral homotopy versus the homotopy, $H$, outlined in Theorem~\ref{thm:hyp}.}
    \label{tab2}
\end{table}

\begin{table}
    \centering
    \begin{tabular}{c|c|c|c|c|c|c|c}
         $n$& $3$ & $4$ &$5$ & $6$ & $7$ & $8$ & $9$   \\
         Polyhedral & $0.03$ & $0.17$ & $1.16$ & $7.04$ & $40.11$ & $228.48$ & $1225.78$  \\
        $H$ & $0.03$ & $0.15$ & $0.83$ & $5.15$ & $34.79$ & $181.11$ & $1027.64$ \\ 
    \end{tabular}
    \caption{Average time (sec) to find all critical points to \eqref{eq:opt} when $d = 4$ using standard polyhedral homotopy versus the homotopy, $H$, outlined in Theorem~\ref{thm:hyp}.}
    \label{tab3}
\end{table}

We implement the homotopy in Theorem~\ref{thm:hyp} with start system \eqref{eq:bin hyp} using the path tracking function in \texttt{HomotopyContinuation.jl}. 
We compare our implementation of the homotopy outlined in Theorem~\ref{thm:hyp} against the polyhedral one in \texttt{HomotopyContinuation.jl} and give the time it takes to run each homotopy algoirthm in Table~\ref{tab1}, Table~\ref{tab2} and Table~\ref{tab3}. The computations are all run using a $2018$ Macbook Pro with 2.3 GHz Quad-Core Intel Core i5.

In all cases, our homotopy algorithm is much faster than the standard off the shelf software. When the hypersurface is of degree two, there are only two complex critical points. Despite this, standard polyhedral homotopy was unable to compute a start system when $n \geq 70$. In contrast, our specialized algorithm was able to find both critical points in a few seconds. We note that in this case, the Bezout bound of the corresponding polynomial system is $2^{n+1}$ where $n$ is the number of variables. When $n = 70$, the Bezout bound is $\approx 2.36 \times 10^{21}$, so it is unreasonable to expect that a total degree homotopy would work in this case.

Similarly, in Table~\ref{tab2} and Table~\ref{tab3} we see that when the degree of the hypersurface is three or four, our algorithm increasingly outperforms the state-of-the-art polyhedral homotopy software as the number of variables increases.

\section{Multiaffine optimization}\label{sec: multi}
 In this final section, we compute the algebraic degree of the following optimization problem:
\begin{align}
    \min_{x \in \mathbb{R}^n} \ g(x) \quad \text{subject to} \quad f(x) = 0 ,\label{eq:multiaffine}
\end{align}
where both $g$ and $f$ are multiaffine, meaning $\newt(f) = \newt(g) = \conv(\{0,1\}^n)$.

\begin{theorem}
The algebraic degree of \eqref{eq:multiaffine} is $!(n+1)$ i.e. the number of derangements 
of $\{0,1,\dots,n\}$. 
\end{theorem}
\begin{proof}
By \cite{ourpaper,rose2022multi} the Lagrange system corresponding to the optimization problem \eqref{eq:multiaffine} is BKK exact. Hence the algebraic degree of \eqref{eq:multiaffine} is equal to the normalized mixed volume of the Newton polytopes of Lagrange system $\mathcal{L}_{g,f} = \{\ell_1,\ldots, \ell_n, f\}$. We denote this value as $\mvol(\mathcal{L}_{g,f})$.

Let us denote by $I_j$ the unit interval $\conv(0, e_j)$ in the $j$-th coordinate direction in $\mathbb{R}^{n+1}$, then the Newton polytope $\newt(\ell_i)$ of $\ell_i$ is given by the Minkowski sum
$$
\newt(\ell_i)= I_0+ I_1 +\ldots + \hat I_i + \ldots + I_n = - I_i + \sum_{j=0}^n I_j.
$$
By definition, the mixed volume of the Newton polytopes of the Lagrange system $\mathcal{L}_{g,f} = \{\ell_1,\ldots, \ell_n, f\}$ is a coefficient in front of the monomial $\lambda_0\lambda_1\cdots\lambda_n$ in the polynomial expansion of 
\begin{equation*}
\begin{aligned}
    &\mathrm{Vol}(\lambda_0\newt(f)+\lambda_1\newt(\ell_1)+\ldots+\newt(\ell_n)) = \\
    & \mathrm{Vol}\left((\Lambda - \lambda_0)\cdot I_0 + (\Lambda - \lambda_1)\cdot I_1 + \ldots + (\Lambda - \lambda_n)\cdot I_n\right),
\end{aligned}
\end{equation*}
where $\Lambda=\sum_{i=0}^n \lambda_i$. A direct computation using multilinearity of mixed volume shows that 
\begin{equation*}
\begin{aligned}
&\mathrm{Vol}\left((\Lambda - \lambda_0)\cdot I_0 + (\Lambda - \lambda_1)\cdot I_1 + \ldots + (\Lambda - \lambda_n)\cdot I_n\right) =\\
& \lambda_0\lambda_1\ldots\lambda_n \sum_{K\subset \{0,\ldots,n\}} (-1)^{|K|}\cdot(n+1-|K|)!  + \text{ other terms.}
\end{aligned}
\end{equation*}
In total, we get the following expression for the mixed volume of the Lagrange system and hence for the algebraic degree of \eqref{eq:multiaffine}:
\begin{equation*}
\begin{aligned}
\mvol(\mathcal{L}_{g,f}) = & \sum_{k=0}^{n+1}(n+1-k)!\cdot(-1)^k\cdot\binom{n+1}{k} \\
= & \sum_{t=0}^{n+1} (t)!\cdot(-1)^{n+1-t}\cdot\binom{n+1}{t} = !(n+1). 
\end{aligned}
\end{equation*}
\end{proof}

\section{Conclusion}
In this paper we presented a homotopy continuation algorithm for finding all complex critical points to a class of polynomial optimization problems. For generic problem parameters, our algorithm is optimal in the sense that it tracks one path for each complex critical point. The main benefit of our work is that we explicitly construct a start system, circumventing the standard bottle neck  associated with polyhedral homotopy algorithms. This advantage was seen in our numerical results which showed that our algorithm was always faster than off-the-shelf homotopy continuation methods and it was able to find all complex critical points when other methods failed. Finally, we concluded by giving an explicit formula for the algebraic degree of a multiaffine polynomial optimization problem.

\bibliographystyle{plain}
\bibliography{refs.bib}

\begin{thebibliography}{10}

\bibitem{aluffiEDComplex}
Paolo Aluffi and Corey Harris.
\newblock The {E}uclidean distance degree of smooth complex projective
  varieties.
\newblock {\em Algebra Number Theory}, 12(8):2005--2032, 2018.

\bibitem{MR3907355}
Carlos Am\'{e}ndola, Nathan Bliss, Isaac Burke, Courtney~R. Gibbons, Martin
  Helmer, Serkan Ho\c{s}ten, Evan~D. Nash, Jose~Israel Rodriguez, and Daniel
  Smolkin.
\newblock The maximum likelihood degree of toric varieties.
\newblock {\em J. Symbolic Comput.}, 92:222--242, 2019.

\bibitem{baaijensRealAlgGroups}
Jasmijn~A. Baaijens and Jan Draisma.
\newblock Euclidean distance degrees of real algebraic groups.
\newblock {\em Linear Algebra Appl.}, 467:174--187, 2015.

\bibitem{bertini-book}
Daniel~J. Bates, Jonathan~D. Hauenstein, Andrew~J. Sommese, and Charles~W.
  Wampler.
\newblock {\em Numerically solving polynomial systems with {B}ertini},
  volume~25 of {\em Software, Environments, and Tools}.
\newblock Society for Industrial and Applied Mathematics (SIAM), Philadelphia,
  PA, 2013.

\bibitem{bernshtein1979the}
David~N. Bernstein.
\newblock The number of roots of a system of equations.
\newblock {\em Funkcional. Anal. i Prilo\v{z}en.}, 9(3):1--4, 1975.

\bibitem{bertsekas2014constrained}
Dimitri~P Bertsekas.
\newblock {\em Constrained optimization and Lagrange multiplier methods}.
\newblock Academic press, 2014.

\bibitem{boyd2004convex}
Stephen Boyd and Lieven Vandenberghe.
\newblock {\em Convex optimization}.
\newblock Cambridge University Press, Cambridge, 2004.

\bibitem{breiding2020euclidean}
P.~Breiding, Frank Sottile, and J.~Woodcock.
\newblock Euclidean distance degree and mixed volume.
\newblock {\em Foundations of Computational Mathematics}, 09 2021.

\bibitem{MR2230921}
Fabrizio Catanese, Serkan Ho\c{s}ten, Amit Khetan, and Bernd Sturmfels.
\newblock The maximum likelihood degree.
\newblock {\em Amer. J. Math.}, 128(3):671--697, 2006.

\bibitem{MR4103774}
Patrick Clarke and David~A. Cox.
\newblock Moment maps, strict linear precision, and maximum likelihood degree
  one.
\newblock {\em Adv. Math.}, 370:107233, 51, 2020.

\bibitem{DraismaTheEDD}
Jan Draisma, Emil Horobe\c{t}, Giorgio Ottaviani, Bernd Sturmfels, and Rekha
  Thomas.
\newblock The {E}uclidean distance degree.
\newblock In {\em S{NC} 2014---{P}roceedings of the 2014 {S}ymposium on
  {S}ymbolic-{N}umeric {C}omputation}, pages 9--16. ACM, New York, 2014.

\bibitem{MR3451425}
Jan Draisma, Emil Horobe\c{t}, Giorgio Ottaviani, Bernd Sturmfels, and Rekha~R.
  Thomas.
\newblock The {E}uclidean distance degree of an algebraic variety.
\newblock {\em Found. Comput. Math.}, 16(1):99--149, 2016.

\bibitem{drusvyatskiyOrthogonally}
Dmitriy Drusvyatskiy, Hon-Leung Lee, Giorgio Ottaviani, and Rekha~R. Thomas.
\newblock The {E}uclidean distance degree of orthogonally invariant matrix
  varieties.
\newblock {\em Israel J. Math.}, 221(1):291--316, 2017.

\bibitem{mvoltext}
G\"{u}nter Ewald.
\newblock {\em Combinatorial convexity and algebraic geometry}, volume 168 of
  {\em Graduate Texts in Mathematics}.
\newblock Springer-Verlag, New York, 1996.

\bibitem{MR2496496}
Hans-Christian Graf~von Bothmer and Kristian Ranestad.
\newblock A general formula for the algebraic degree in semidefinite
  programming.
\newblock {\em Bull. Lond. Math. Soc.}, 41(2):193--197, 2009.

\bibitem{MR2988436}
Elizabeth Gross, Mathias Drton, and Sonja Petrovi\'{c}.
\newblock Maximum likelihood degree of variance component models.
\newblock {\em Electron. J. Stat.}, 6:993--1016, 2012.

\bibitem{HostenSolving}
Serkan Ho\c{s}ten, Amit Khetan, and Bernd Sturmfels.
\newblock Solving the likelihood equations.
\newblock {\em Found. Comput. Math.}, 5(4):389--407, 2005.

\bibitem{huber1995a}
Birkett Huber and Bernd Sturmfels.
\newblock A polyhedral method for solving sparse polynomial systems.
\newblock {\em Math. Comp.}, 64(212):1541--1555, 1995.

\bibitem{MR3103064}
June Huh.
\newblock The maximum likelihood degree of a very affine variety.
\newblock {\em Compos. Math.}, 149(8):1245--1266, 2013.

\bibitem{Khachiyan1993}
Leonid Khachiyan.
\newblock {\em Complexity of Polytope Volume Computation}, pages 91--101.
\newblock Springer Berlin Heidelberg, Berlin, Heidelberg, 1993.

\bibitem{khovanskii1978newton}
Askold~G. Khovanskii.
\newblock Newton polyhedra, and the genus of complete intersections.
\newblock {\em Funktsional. Anal. i Prilozhen.}, 12(1):51--61, 1978.

\bibitem{kouchnirenko1976polyedres}
Anatoli~G. Kouchnirenko.
\newblock Poly\`edres de {N}ewton et nombres de {M}ilnor.
\newblock {\em Invent. Math.}, 32(1):1--31, 1976.

\bibitem{lasserre2000global}
Jean~B. Lasserre.
\newblock Global optimization with polynomials and the problem of moments.
\newblock {\em SIAM J. Optim.}, 11(3):796--817, 2000/01.

\bibitem{leeFermat}
Hwangrae Lee.
\newblock The {E}uclidean distance degree of {F}ermat hypersurfaces.
\newblock {\em J. Symbolic Comput.}, 80(part 2):502--510, 2017.

\bibitem{ourpaper}
Julia Lindberg, Leonid Monin, and Kemal Rose.
\newblock Algebraic degree of sparse polynomial optimization.
\newblock {\em preprint}, 2022.

\bibitem{lindberg2021the}
Julia Lindberg, Nathan Nicholson, Jose~Israel Rodriguez, and Zinan Wang.
\newblock The maximum likelihood degree of sparse polynomial systems, 2021.

\bibitem{invariants2022lindberg}
Julia Lindberg and Jose Rodriguez.
\newblock Invariants of sdp exactness in quadratic programming, 2022.

\bibitem{manivel}
Laurent Manivel, Mateusz Micha{\l}ek, Leonid Monin, Tim Seynnaeve, and Martin
  Vodi{\v{c}}ka.
\newblock Complete quadrics: Schubert calculus for gaussian models and
  semidefinite programming.
\newblock {\em arXiv preprint arXiv:2011.08791}, 2020.

\bibitem{maximMultiview}
Laurentiu~G. Maxim, Jose~I. Rodriguez, and Botong Wang.
\newblock Euclidean distance degree of the multiview variety.
\newblock {\em SIAM J. Appl. Algebra Geom.}, 4(1):28--48, 2020.

\bibitem{MR4219257}
Mateusz Micha{\l}ek, Leonid Monin, and Jaros\l aw~A. Wi\'{s}niewski.
\newblock Maximum likelihood degree, complete quadrics, and
  {$\mathbb{C}^*$}-action.
\newblock {\em SIAM J. Appl. Algebra Geom.}, 5(1):60--85, 2021.

\bibitem{molzahn2019a}
Daniel~K. Molzahn and Ian~A. Hiskens.
\newblock A survey of relaxations and approximations of the power flow
  equations.
\newblock {\em Foundations and Trends in Electric Energy Systems},
  4(1-2):1--221, 2019.

\bibitem{MR2507133}
Jiawang Nie and Kristian Ranestad.
\newblock Algebraic degree of polynomial optimization.
\newblock {\em SIAM J. Optim.}, 20(1):485--502, 2009.

\bibitem{nie2009algebraic}
Jiawang Nie and Kristian Ranestad.
\newblock Algebraic degree of polynomial optimization.
\newblock {\em SIAM Journal on Optimization}, 20(1):485--502, 2009.

\bibitem{MR2546336}
Jiawang Nie, Kristian Ranestad, and Bernd Sturmfels.
\newblock The algebraic degree of semidefinite programming.
\newblock {\em Math. Program.}, 122(2, Ser. A):379--405, 2010.

\bibitem{parrilo2003semidefinite}
Pablo~A. Parrilo.
\newblock Semidefinite programming relaxations for semialgebraic problems.
\newblock volume~96, pages 293--320. 2003.
\newblock Algebraic and geometric methods in discrete optimization.

\bibitem{poljak1995a}
S.~Poljak, F.~Rendl, and H.~Wolkowicz.
\newblock A recipe for semidefinite relaxation for {$(0,1)$}-quadratic
  programming.
\newblock {\em J. Global Optim.}, 7(1):51--73, 1995.

\bibitem{potra2000interior}
Florian~A. Potra and Stephen~J. Wright.
\newblock Interior-point methods.
\newblock volume 124, pages 281--302. 2000.
\newblock Numerical analysis 2000, Vol. IV, Optimization and nonlinear
  equations.

\bibitem{MR3686780}
Jose~Israel Rodriguez and Botong Wang.
\newblock The maximum likelihood degree of mixtures of independence models.
\newblock {\em SIAM J. Appl. Algebra Geom.}, 1(1):484--506, 2017.

\bibitem{rose2022multi}
Kemal Rose.
\newblock Multi-degrees in polynomial optimization.
\newblock {\em arXiv preprint arXiv:2209.10670}, 2022.

\bibitem{sommese2005numerical}
A.J. Sommese and C.W. Wampler.
\newblock {\em The Numerical Solution Of Systems Of Polynomials Arising In
  Engineering And Science}.
\newblock World Scientific Publishing Company, 2005.

\bibitem{Sturmfels-CBMS}
Bernd Sturmfels.
\newblock {\em Solving systems of polynomial equations}, volume~97 of {\em CBMS
  Regional Conference Series in Mathematics}.
\newblock Published for the Conference Board of the Mathematical Sciences,
  Washington, DC; by the American Mathematical Society, Providence, RI, 2002.

\bibitem{MR4196404}
Bernd Sturmfels, Sascha Timme, and Piotr Zwiernik.
\newblock Estimating linear covariance models with numerical nonlinear algebra.
\newblock {\em Algebr. Stat.}, 11(1):31--52, 2020.

\bibitem{tan2001the}
Peng~Hui Tan and L.K. Rasmussen.
\newblock The application of semidefinite programming for detection in {CDMA}.
\newblock {\em IEEE Journal on Selected Areas in Communications},
  19(8):1442--1449, 2001.

\bibitem{vavasis1990quadratic}
Stephen~A. Vavasis.
\newblock Quadratic programming is in {NP}.
\newblock {\em Inform. Process. Lett.}, 36(2):73--77, 1990.

\bibitem{wang2021tssos}
Jie Wang, Victor Magron, and Jean-Bernard Lasserre.
\newblock T{SSOS}: a moment-{SOS} hierarchy that exploits term sparsity.
\newblock {\em SIAM J. Optim.}, 31(1):30--58, 2021.

\end{thebibliography}

\end{document}